\documentclass[11pt]{amsart}
\usepackage{epsfig}
\usepackage{amssymb}
\usepackage{amsthm}
\usepackage[all,knot,arc]{xy}
\usepackage{color}
\usepackage{hyperref}
\hypersetup{colorlinks=true,citecolor=blue,linkcolor=blue}
\newtheorem{theorem}{Theorem}[section]
\newtheorem{lemma}[theorem]{Lemma}

\newtheorem{corollary}[theorem]{Corollary}

 \usepackage{palatino}

\theoremstyle{definition}

\newtheorem{remark}[theorem]{Remark}

\def\dfn#1{{\em #1}}
\DeclareMathOperator{\tb}{tb}
\DeclareMathOperator{\rot}{r}
\renewcommand{\sl}{\ensuremath{{\, sl}}}

\newcommand{\R}{\ensuremath{\mathbb{R}}}
\newcommand{\Z}{\ensuremath{\mathbb{Z}}}

\numberwithin{equation}{section}

\def\mathbi#1{\textbf{\em#1}}

\title{On the Legendrian and Transverse classification of Cablings}

\author{B\"{u}lent Tosun}
\address{School of Mathematics \\ Georgia Institute of Technology}
\email{btosun3@math.gatech.edu}
\urladdr{\href{http://www.math.gatech.edu/users/btosun3}{http://www.math.gatech.edu/users/btosun3}}

\begin{document}

\maketitle

\begin{abstract}

 In this note we study Legendrian and transverse knots in the knot type of a $(p,q)$-cable of a knot $\mathcal{K}$ in $S^3$. We give two structural theorems that describe when the $(p,q)$- cable of a {\em Legendrian simple} knot type $\mathcal{K}$ is also {\em Legendrian simple}.  
 
\end{abstract}

 \section{Introduction}
Legendrian and transverse knots are not just natural objects of study in contact $3$--manifolds but also important in capturing the geometry and topology of underlying contact structure. For example, tight vs. overtwisted dichtomy is a result of having a bound on the classical invariants $\tb(L)$ and $\rot(L)$ associated to a Legendrian knot $L$ in tight contact structures, see \cite{{BakerEtnyreVanHornMorris10},{be},{el}}. A similar statement is true for transverse knots in a given knot type and  for their unique classical invariant, self-linking number $\sl$. Hence, one wants to better understand the classification of Legendrian and transverse knots. In particular, one naturally  wonder if $\tb$ and $\rot$ (respectively $\sl$) determine the Legendrian (respectively transverse) isotopy class completely. Such a knot type is called Legendrian (respectively transverse) simple and  non-simple otherwise. This problem has been worked out on some nice class of knots \cite{{ef},{eh2}, {EtnyreNgVertesi}} and under certain topological operation on certain knot types ~\cite{{eh1},{EtnyreHonda03}}. In this paper we continue the study of cabling begun in~\cite{eh1}.

\subsection{Cabling.} Recall the $(p,q)$-cable of a knot type $\mathcal K$, denoted $\mathcal K_{(p,q)}$, is the knot type obtained by taking the isotopy class of a $(p,q)$-curve on the boundary of a tubular neighborhood of a representative of $\mathcal{K}$ (where $p$ denotes the meridional winding and $q$ denotes longitudinal winding). We will also denote this $(p,q)$-curve by the fraction $\frac qp$.

 In \cite {eh1}, Etnyre and Honda studied the Legendrian and transverse classification of cables of a knot in $(S^3,\xi_{std})$ that satisfy a property called the \textit{uniform thickness property} (UTP). They proved that $\mathcal{K}_{(p,q)}$ is Legendrian simple if $\mathcal{K}$ is Legendrian simple and uniformly thick. The UTP is, for example, satisfied by negative torus knots \cite{eh1} and is known to be preserved under cabling operation \cite{eh1}, \cite{LaFountain1}. On the other hand, the unknot and positive torus knots \cite{{eh1},{EtnyreLaFTosun}} are some examples of non-uniformly thick knot types. Indeed, by using the fact that the $(2,3)$-torus knot is not uniformly thick, Etnyre and Honda exhibit one of the first examples of a transversely non-simple knot type, the $(2,3)$-cable of the $(2,3)$-torus knot ($\it{cf.}$ \cite {bm}). Finally, in \cite{EtnyreLaFTosun}, Legendrian and transverse cables of positive torus knots were completely classified, using, in part, the results in this paper.

{\em Aim:} In this paper we study Legendrian and transverse simplicity for cables of a knot type $\mathcal K$ which is not necessarily uniformly thick. This assumption led us to study two knot invariants the \dfn{contact width} and the \dfn{lower width}, the first of which was already introduced and studied in \cite{eh1}. 

\subsection{The contact width and sufficiently positive cables.} Given a tight contact manifold $(M,\xi)$, let $\mathcal{K}$ be a topological knot type and  $\mathcal L(\mathcal K)$ be the set of Legendrian isotopy classes of $\mathcal K$. As the Thurston-Bennequin invariant of a knot $L$ in $\mathcal L(\mathcal K)$ is bounded above \cite{{be},{el}}, we may then define the maximal Thurston-Bennequin number of a knot type $\mathcal K$ to be

\[
\overline{ \tb} (\mathcal {K}) =\max \{\tb(L) |\ L \in {\mathcal L} (\mathcal K) \}.
\]

\smallskip

The contact width of a knot type is given by
\[   
\omega(\mathcal K)=\sup\frac{1}{slope(\Gamma_{\partial (S^1 \times D^2)})},
\]
where the supremum is taken over all $S^1 \times D^2\hookrightarrow S^3$ representing $\mathcal K$ with $\partial (S^1 \times D^2)$ convex. In order to make  sense of slopes of homotopically non trivial curves on $\partial (S^1 \times D^2)$ we identify $\partial (S^1 \times D^2)=\mathbb {R}^2/\mathbb {Z}^2$ where the meridian has slope $0$ and the well-defined longitude (as $\mathcal K$ is in $S^3$) has slope $\infty$.

Our first main result can now be stated as follows.

\begin{theorem}\label{thm:positive}
If $\mathcal{K}$ is Legendrian simple and $\omega (\mathcal{K})\in \mathbb{Z}$. Then its $(p,q)$-cable, $\mathcal{K}_{(p,q)}$, is also Legendrian simple and admits a classification in terms of the classification of $\mathcal{K}$, provided $\frac{p}{q}>\omega (\mathcal{K})$. Moreover the maximal Thurston-Bennequin invariant is $$\overline{ \tb} (\mathcal {K}_{(p,q)}) = pq - | \overline{ \tb}(\mathcal {K}) \bullet \frac{p}{q} |,$$ and the set of rotation numbers associated to $L \in \mathcal K_{(p,q)} $ with  $ \tb (L) = \overline{ \tb} (\mathcal K_{(p,q)})$ is 
\[
\rot (L) = \{q\cdot \rot(K) |\ K \in {\mathcal L} (\mathcal K)\ , \mbox{}  \tb (K) = \overline{\tb} (\mathcal K) \}
\]
 If $K\in\mathcal L(\mathcal K)$ is a non-destabilizable with $ \tb(K)=n<\overline{ \tb}(K)$, then there is  non-destabilizable $L$ in $\mathcal L(\mathcal K_{(p,q)})$ with $ \tb(L)=pq-| \frac{1}{n} \bullet \frac{p}{q} |$ and the set of rotation numbers associated to non-destabilizable $L \in \mathcal K_{(p,q)} $ with  $ \tb (L) =pq-| \frac{1}{n} \bullet \frac{p}{q} | $ is 
\[
\rot (L) = \{q\cdot \rot(K) |\ K \in {\mathcal L} (\mathcal K)\ , \  \tb (K) = n \}. 
\]
\end{theorem}

\subsection{The UTP, the lower width and sufficiently negative cables.} A knot type $\mathcal K$ is said to satisfy the uniform thickness property if the following hold:
 
 \begin{enumerate}
 
 \item $\overline  \tb(\mathcal K)=\omega(\mathcal K)$
 \item  Every embedded solid tori $S^1 \times D^2\hookrightarrow S^3$ representing $\mathcal K$ can be thickened to a standard neighborhood of a maximal $ \tb$ Legendrian knot.
 \end{enumerate}

The motivation behind this definition the following result of Etnyre and Honda
\begin{theorem}[Etnyre and Honda 2005,\cite{eh1}]\label{EH}
If $\mathcal K$ is knot type which is Legendrian simple and satisfies the UTP, then $\mathcal K_{(p,q)}$ is Legendrian simple for any $p,q$.
\end{theorem}

 We say that a solid torus $S^1\times D^2$ with convex boundary representing $\mathcal K$ is \dfn{non-thickenable}, if there is no $N'$ containing $S^1\times D^2$ (whenever we discuss solid torus contained in another we assume they have the same core) with $slope(\Gamma_{N'})\neq slope(\Gamma_{N})$. Since there are knots with this property, see \cite{{eh1},{EtnyreLaFTosun}}, we define yet another invariant of a Legendrian knot, the {\em lower width}, to be
\[
\ell w(\mathcal K) = \textrm{inf}\frac{1}{\textrm{slope}(\Gamma_{\partial (S^1 \times D^2)})},
\]
where $S^1\times D^2$ ranges over all non-thickenable solid tori representing $\mathcal K$ with convex boundary.     
 
\smallskip 
Our second main result addresses classification of cables with cabling slope less than the lower width.     

\begin{theorem}\label{thm:lower}
If $\mathcal{K}$ is Legendrian simple and $\ell \omega (\mathcal{K})\in\mathbb{Z}$. Then $\mathcal{K}_{(p,q)}$ is also Legendrian simple, provided $\frac{p}{q}<\ell \omega (\mathcal{K})$. Moreover $$ \displaystyle \overline{ \tb}(\mathcal{K}_{(p,q)})=pq=\omega(\mathcal{K}_{(p,q)}),$$
and the set of rotation numbers realized by$$ \{L_{(p,q)}\in\mathcal{L}(\mathcal{K}_{(p,q)}): \tb(L)=\displaystyle\overline{ \tb}(\mathcal{K}_{(p,q)})\}$$ is $$
   \{\pm(p+q(n+r(L)):L\in\mathcal{L}(\mathcal{K}), \tb(L)=-n\}$$ where n is the integer that satisfies $$-n-1<\frac{p}{q}<-n.$$

\end{theorem}

 It is not difficult to show for any knot type $\mathcal K$ the inequality $\overline{ \tb}(\mathcal{K})\leq\omega (\mathcal{K})\leq \overline{ \tb}(\mathcal{K})+1$ holds. Now, if $\mathcal K$ is the unknot, then $\omega (\mathcal{K})=\displaystyle \overline{ \tb}(\mathcal{K})+1=0$ (since $\displaystyle \overline{ \tb}(\mathcal{K})=-1$ and complement of the unknot
in $S^3$ is the neighborhood of an unknot again). Also note that $\ell\omega(\mathcal K)=\infty$. Hence, in the case that $\mathcal K$ is the unknot in Theorem ~ \ref{thm:positive} and ~\ref{thm:lower} we get the following corollary which was originally proved by Etnyre and Honda in \cite{eh2}.

\begin{corollary}
Torus knots are Legendrian, and hence transversely simple. 
\end{corollary}

Moreover, if $\omega(\mathcal K)=\overline{\tb}(\mathcal K)= \ell\omega(\mathcal K)\in \mathbb{Z}$, then $\mathcal K$ is uniformly thick and our Theorems \ref{thm:positive} and \ref{thm:lower} recover Theorem~\ref{EH} of Etnyre and Honda above.
 
 \smallskip
 
\subsection{Idea of the proof and outline} The main idea is to show under the assumptions of each of the theorem above, one understands solid tori in a given knot type well enough to classify maximal $\tb$ representatives of the cabled knot type and can always find bypass disks and hence destabilize a Legendrian knot with mon-maximal $\tb$. The necessary assumptions in the theorems was mainly motivated from the work of Etnyre and Honda in ~\cite{eh1}.  We begin, in Section $2$, by collecting the necessary definitions, tools and facts about convex surface theory \cite{{Colin97},{gi1}, {h1}} and about the classification of Legendrian and transverse knots \cite{eh2}. With these definitions in place, we conclude, in Section $3$, with the proof of Theorem ~ \ref{thm:positive} and Theorem ~ \ref{thm:lower}.    

\smallskip

\textit {Acknowledgement.}
  The author is grateful to his advisor \mbox{John Etnyre} for his continous support and guidence. The author would like to thank Douglas LaFountain, Vera Vertesi and Lenny Ng for helpful discussions. The author partially supported during the course of this work by T\"{U}B\.{I}TAK, the Scientific and Technological Research Council of Turkey and NSF Grant DMS-0804820. 
\section{Preliminaries}\label{sec:pre}

 In this section we will give  basic definitions and the necessary background material which will be used in the rest of the paper. 

\subsection{Convex surfaces, bypasses and the Farey tessellation}
Recall a surface $\Sigma$ in a contact manifold $(M,\xi)$ is
\dfn{convex} if it has a neighborhood $\Sigma\times I$, where
$I=(-\epsilon, \epsilon)$ is some interval, and $\xi$ is $I$-invariant
in this neighborhood. Any closed surface can be $C^\infty$-perturbed
to be convex. Moreover if $L$ is a Legendrian knot on $\Sigma$ for
which the contact framing is non-positive with respect to the framing
given by $\Sigma$, then $\Sigma$ may be perturbed in a $C^0$ fashion
near $L$, but fixing $L$, and then again in a $C^\infty$ fashion away
from $L$ so that $\Sigma$ is convex.

Given a convex surface $\Sigma$ with $I$-invariant neighborhood, let
$\Gamma_\Sigma \subset \Sigma$ be the multicurve where $\xi$ is
tangent to the $I$ factor. This is called the \dfn{dividing set} of
$\Sigma.$ If $\Sigma$ is oriented it is easy to see that
$\Sigma\setminus \Gamma=\Sigma_+\cup \Sigma_-$ where $\xi$ is
positively transverse to the $I$ factor along $\Sigma_+$ and
negatively transverse along $\Sigma_-$. If $L$ is a Legendrian curve
on a $\Sigma$ then the framing of $L$ given by the contact planes,
relative to the framing coming from $\Sigma$, is given by $-\frac 12
(L\cdot \Gamma)$. Moreover if $L=\partial \Sigma$ then the rotation
number of $L$ is given by $\rot(L)=\chi(\Sigma_+)-\chi(\Sigma_-)$.

\subsubsection{Convex tori}
A convex torus $T$ is said to be in \dfn{standard form} if $T$ can be
identified with $\R^2/\Z^2$ so that $\Gamma_T$ consists of $2n$
horizontal curves (note $\Gamma_T$ will always have an even number of
curves and we can choose a parameterization to make them horizontal) and
the characteristic foliations consists of $2n$ vertical lines of
singularities ($n$ lines of sources and $n$ lines of sinks) and the
rest of the foliation is by non-singular lines of slope $s$. See Figure~\ref{fig:ct}.

 \begin{figure}[h!]
\begin{center}
  \includegraphics[width=4cm]{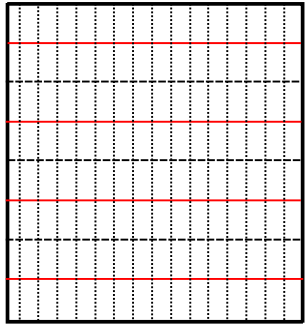}
 \caption{Standard convex tori. Legendrian divides are dashed. Legendrian rulings are dotted vertical, i.e. non-singular lines of slope $\infty$.}
  \label{fig:ct}
\end{center}
\end{figure}
 
The lines of singularities are called \dfn{Legendrian divides} and the
other curves are called \dfn{ruling curves}. We notice that the Giroux
Flexibility Theorem allows us to $C^0$ isotope any convex torus into
standard form.
 
\subsubsection{Bypasses and tori}
Let $\Sigma$ be a convex surface and $\alpha$ a Legendrian arc in
$\Sigma$ that intersects the dividing curves $\Gamma_\Sigma$ in 3
points $p_1,p_2,p_3$ (where $p_1, p_3$ are the end points of the
arc). Then a \dfn{bypass for $\Sigma$ (along $\alpha$)}, 
is a convex disk $D$ with Legendrian boundary
such that
\begin{enumerate}
\item $D\cap \Sigma=\alpha,$
\item $tb(\partial D)=-1,$
\item $\partial D= \alpha\cup \beta,$
\item $\alpha\cap\beta=\{p_1,p_3\}$ are corners of $D$ and elliptic
singularities of $D_\xi.$
\end{enumerate}
The bypass attachment operation is the basic unit of isotopy of surfaces and will be crucial in our proofs. It is given
in the following theorem.
\begin{theorem}[Honda 2000, \cite{h1}]
 Let $\Sigma$ be a convex surface, $D$ a bypass for $\Sigma$ along vertical $\alpha$ in $\Sigma$ (Figure~\ref{fig:ct}), then there exists a neighborhood of $\Sigma\cup D\subset M$ diffeomorphic to $\Sigma\times [0,1]$, such that $\Sigma=\Sigma_{0}$, $\Sigma_{1}$ are convex, $\Sigma\times [0,\epsilon]$ is $I$--invariant and $\Gamma_{\Sigma}$ is related to $\Gamma_{\Sigma_{1}}$ as in Figure~\ref{fig:ba}. 
\end{theorem}

\begin{figure}[h!]
\begin{center}
  \includegraphics[width=7cm]{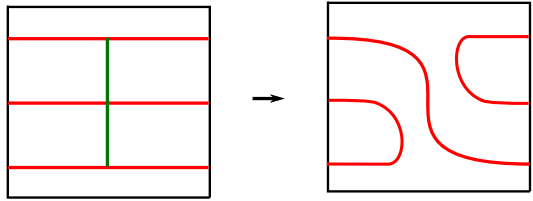}
 \caption{A piece of original surface $\Sigma$ with bypass arc $\alpha$, on the left. The surface $\Sigma_{1}$ after isotoping $\Sigma$ across $D$, on the right.}
  \label{fig:ba}
\end{center}
\end{figure}

A surface $\Sigma$ locally separates the ambient manifold. If a bypass
is contained in the (local) piece of $M\setminus \Sigma$ that has
$\Sigma$ as its oriented boundary then we say the bypass will be
attached to the back of $\Sigma$ otherwise we say it is attached to
the back of $\Sigma$.

When a bypass is attached to a torus $T$ then either the dividing
curves do not change and their number increases by two, or decreases by
two, or the slope of the dividing curves changes. The slope of the
dividing curves can change only when there are two dividing
curves. If the bypass is
attached to $T$ along a ruling curve then either the number of
dividing curves decreases by two or the slope of the dividing curves
changes. To understand the change in slope we need the following.  Let
$\mathbb{D}$ be the unit disk in $\R^2.$ Recall the \dfn{Farey
  tessellation} of $\mathbb{D}$ is constructed as follows.  Label the
point $(1,0)$ on $\partial \mathbb{D}$ by $0=\frac01$ and the point
$(-1,0)$ with $\infty=\frac10.$ Now join them by a geodesic.  If two
points $\frac{p}{q},$ $\frac{p'}{q'}$ on $\partial \mathbb{D}$ with
non-negative $y$-coordinate have been labeled then label the point on
$\partial\mathbb{D}$ half way between them (with non-negative
$y$-coordinate) by $\frac{p+p'}{q+q'}.$ Then connect this point to
$\frac{p}{q}$ and to $\frac{p'}{q'}$ by a hyperbolic
geodesic. Continue this until all positive fractions have been
assigned to points on $\partial \mathbb{D}$ with non-negative
$y$-coordinates. Now repeat this process for the points on $\partial
\mathbb{D}$ with non-positive $y$-coordinate except start with
$\infty=\frac{-1}{0}.$ See Figure~\ref{fig:bf}.

 \begin{figure}[h!]
\begin{center}
  \includegraphics[width=11cm]{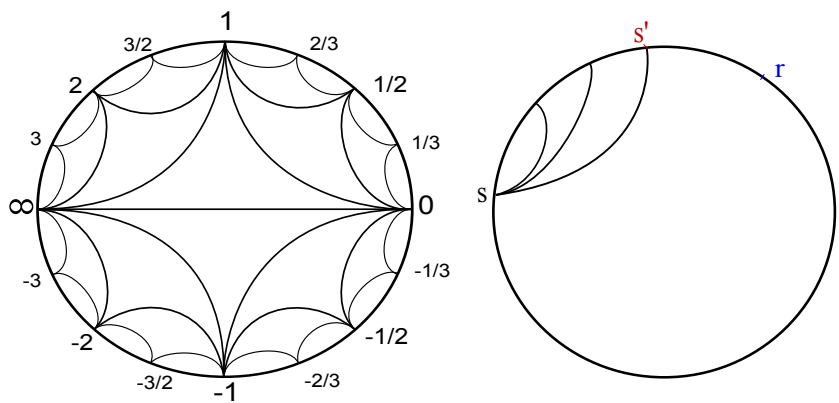}
 \caption{The Farey tessellation, on the left. Schematic of the change in the dividing slope from $s$ to $s'$ after bypass attachment along a Legendrian ruling curve of slope $r$ on the Farey tessellation, on the right.}
  \label{fig:bf}
\end{center}
\end{figure}

The key result we need to know about the Farey tessellation is given
in the following theorem. See Figure~\ref{fig:bf}.
\begin{theorem}[Honda 2000, \cite{h1}]\label{bpaferry}
Let $T$ be a convex torus in standard form with $|\Gamma_{T}|=2,$
dividing slope $s$ and ruling slope $r\not=s.$ Let $D$ be a bypass for
$T$ attached to the front of $T$ along a ruling curve. Let $T'$ be the
torus obtained from $T$ by attaching the bypass $D.$ Then
$|\Gamma_{T'}|=2$ and the dividing slope $s'$ of $\Gamma_{T'}$ is
determined as follows: let $[r,s]$ be the arc on $\partial\mathbb{D}$
running from $r$ counterclockwise to $s,$ then $s'$ is the point in
$[r,s]$ closest to $r$ with an edge to $s.$

If the bypass is attached to the back of $T$ then the same algorithm
works except one uses the interval $[s,r]$ on
$\partial\mathbb{D}$. \qed
\end{theorem}

\subsubsection{The Imbalance Principle}
As we see that bypasses are useful in changing dividing curves on a
surface we mention a standard way to try to find them called the
Imbalance Principle. Suppose that $\Sigma$ and $\Sigma'$ are two
disjoint convex surfaces and $A$ is a convex annulus whose interior is
disjoint from $\Sigma$ and $\Sigma'$ but its boundary is Legendrian
with one component on each surface. If $|\Gamma_\Sigma\cdot \partial
A|>|\Gamma_{\Sigma'}\cdot A|$ then there will be a dividing curve on
$A$ that cuts a disk off of $A$ that has part of its boundary on
$\Sigma$. It is now easy to use the Giroux Flexibility Theorem to show
that there is a bypass for $\Sigma$ on $A$.

\subsubsection{Discretization of Isotopy}\label{sssec:discritize}
We will frequently need to analyze what happens to the contact
geometry when we have a topological isotopy between two convex
surfaces $\Sigma$ and $\Sigma'$. This can be done by the technique of
\dfn{Isotopy Discretization} \cite{Colin97} (see also
\cite{eh2} for its use in studying Legendrian knots). Given
an isotopy between $\Sigma$ and $\Sigma'$ one can find a sequence of
convex surfaces $\Sigma_1=\Sigma, \Sigma_2, \ldots, \Sigma_n=\Sigma'$
such that
\begin{enumerate}
\item all the $\Sigma_i$ are convex and
\item $\Sigma_i$ and $\Sigma_{i+1}$ are disjoint and $\Sigma_{i+1}$ is
obtained from $\Sigma_i$ by a bypass attachment.
\end{enumerate}
Thus if one is trying to understand how the contact geometry of
$M\setminus \Sigma$ and $M\setminus \Sigma'$ relate, one just needs to
analyze how the contact geometry of the pieces of $M\setminus
\Sigma_i$ changes under bypass attachment. In particular, many
arguments can be reduced from understanding a general isotopy to
understanding an isotopy between two surfaces that cobound a product
region.

There is also a relative version of Isotopy Discretization where
$\Sigma$ and $\Sigma'$ are convex surfaces with Legendrian boundary
consisting of ruling curves on a convex torus. If $\partial
\Sigma=\partial \Sigma'$ and there is a topological isotopy of
$\Sigma$ to $\Sigma'$ relative to the boundary then we can find a
discrete isotopy as described above. (Note that during the discrete
isotopy the boundary of the surface is not fixed but is allowed to
move among the ruling curves on the convex torus. One could 
slightly rephrase item (2) in the above definition of a discretized isotopy to keep the boundary fixed, but we find it more natural to allow the boundary to move even though the original isotopy is relative to the boundary.)

\subsection{Standard neighborhood, transverse knots and stable simplicity.} 
Given a Legendrian knot $L$, a \dfn{standard neighborhood} of $L$ is a
solid torus $N$ that has convex boundary with two dividing curves of
slope $1/\tb(L)$ (and of course we will usually take $\partial N$ to
be a convex torus in standard form). Conversely given any such solid
torus it is a standard neighborhood of a unique Legendrian knot ($\it{cf.}$ \cite{Kanda}).

 One may understand stabilizations and destabilizations of a Legendrian
knot $L$ in terms of the standard neighborhood. Specifically, inside
the standard neighborhood $N$ of $L$, $L$ can be positively stabilized
to $S_+(L)$, or negatively stabilized to $S_-(L)$. Let $N_\pm$ be a
neighborhood of the stabilization of $L$ inside $N.$ As above we can
assume that $N_\pm$ has convex boundary in standard form. It will have
dividing slope $\frac{1}{\tb(L)-1}.$ Thus the region $N\setminus
N_\pm$ is diffeomorphic to $T^2\times[0,1]$ and the contact structure
on it is easily seen to be a \dfn{basic slice}, see
\cite{h1}. There are exactly two basic slices with given
dividing curves on their boundary and as there are two types of
stabilization of $L$ we see that the basic slice $N\setminus N_\pm$ is
determined by the type of stabilization done, and vice versa. Moreover
if $N$ is a standard neighborhood of $L$ then $L$ destabilizes if the
solid torus $N$ can be thickened to a solid torus $N_d$ with convex
boundary in standard form with dividing slope $\frac 1{\tb(L)+1}.$
Moreover the sign of the destabilization will be determined by the
basic slice $N_d\setminus N$. Finally, we notice that using
Theorem~\ref{bpaferry} we can destabilize $L$ by finding a bypass for
$N$ attached along a ruling curve whose slope is clockwise of
$1/(\tb(L)+1)$ (and anti-clockwise of $0$).

 Furthermore, by using this neighborhood one can talk about the \dfn{positive/negative transverse push-off}, $T_{\pm}(L)$ of a Legendrian knot $L$.  The only classical invariant of these transverse knots, the self linking number, can be computed for transverse push-offs as ($\it{cf.}$ \cite {EpsteinFuchsMeyer01}) 
\[
\sl(T_{\pm}(L))= \tb(L)\mp \rot(L). 
\]

 As in \cite {eh2} two  Legendrian knots $L$ and $L'$  are called \textit{stably isotopic} if there is some $n$ and $n'$ such that $S^n_{-}(L)$ and $S^{n'}_{-}(L')$ are Legendrian isotopic. Note that $ \tb(L)-\rot(L)= \tb(S_{-}(L))-\rot(S_{-}(L))$. A knot type $\mathcal K$ is called \textit{stably simple} if Legendrian knots in this knot type are stably isotopic. The key result that we need concerning the transverse classification of a knot type is the following theorem of Epstein, Fuchs and Meyer from \cite{EpsteinFuchsMeyer01} (also \cite {eh2} for general manifolds) which reduces the classification of transverse knots up to transverse isotopy to the classification Legendrian knots up to Legendrian isotopy and their negetaive stabilizations.

\begin{theorem}[Epstein-Fuchs-Meyer \cite{EpsteinFuchsMeyer01}, Etnyre-Honda \cite{eh2}]\label{stably simple}
A knot type $\mathcal K$ is stably simple if and only if it is transversely simple.
\end{theorem}
 
\subsection{Framings and the strategy of the proofs} 
One can talk about two coordinate systems for the boundary of a neighborhood of  $\mathcal K_{(p,q)}$. In the first coordinate system, denoted by $\mathcal C$, the meridian has slope $0$ and the well-defined longitude (coming from the intersection of a Seifert surface for $\mathcal K_{(p,q)}$ with $\partial N(\mathcal K_{(p,q)})$) has slope $\infty$. In the second coordinate system, denoted $\mathcal C'$, the meridian has slope $0$ and slope $\infty$ comes from the surface $\partial N(\mathcal K)$ on which $\mathcal K_{(p,q)}$ sits. 
As explained in \cite {eh1} one can relate these two framings for $\partial N(\mathcal K_{(r,s)})$ and deduce the following relation between the twisting of the contact planes along  $L_{(r,s)}$ in $\mathcal K_{(r,s)}$ with respect to either framings.

 \begin{equation}\label{twisting}
 t(L_{(p,q)}, \mathcal{C}')+pq=t(L_{(p,q)},\mathcal{ C})= \tb(L_{(p,q)}).
 \end{equation}
 
 Given two embedded closed curves $\gamma$ and $\gamma'$ on a  torus $T$ we denote their minimal intersection by $\gamma\bullet\gamma'$. If the slope of $\gamma$, respectively $\gamma'$, is $s=\frac{r}{t}$, respectively $s'=\frac{r'}{t'}$, then $$s \bullet s'=|rt'-tr'|.$$
  
  The following two lemmas are from \cite{eh1}. The first one is an easy consequence of Equation~\eqref{twisting} above.  

\begin{lemma} Let $L_{(p,q)}$ be a Legendrian knot in the knot type of $\mathcal K_{(p,q)}$.

 \begin{enumerate}
   \item If $L_{(p,q)}$ is a Legendrian divide on $\partial N(\mathcal K)$ (in which case $slope(\Gamma_{\partial N(\mathcal K)})=\frac{q}{p}$), then
    $$\tb(L_{(p,q)})=pq.$$
   
  \item If $L_{(p,q)}$ is a Legendrian ruling curve on $\partial N(\mathcal K)$ and $slope(\Gamma_{\partial N(\mathcal K)})=\frac{q'}{p'}$, then
   $$ \tb(L_{(p,q)})=pq-|\frac{p}{q}\bullet\frac{p'}{q'}|.$$

 \end{enumerate}
\end{lemma}  
  
 \begin{lemma}\label{rotation}
 Let $D$ be a convex disk contained in $N(\mathcal K)$ with Legendrian boundary on a contact isotopic copy of the convex surface $\partial N(\mathcal K)$ and $\partial\Sigma(L)$ is a convex Seifert surface of a Legendrian knot $L\in\mathcal L(\mathcal K)$ which is contained in a contact isotopic copy of $\partial N(\mathcal K)$. Then
 \begin{equation}\label{rottt}
  \rot(L_{(p,q)})=p\cdot \rot(\partial D)+q\cdot \rot(\partial\Sigma(L)).
 \end{equation}
 \end{lemma}

 \section{Legendrian simple cables}
   
    In this section we give the proofs of Theorem \ref{thm:positive} and Theorem \ref{thm:lower}. But first we want to note that in the proofs we will, impicitly, use the following classical strategy, first proposed by Etnyre in ~\cite{e} and efficiently used for almost all known results concerning the clasification of Legendrian knots.  
  \begin{enumerate}
  \item Find a formula that computes $\overline{ \tb}(\mathcal{K}_{(p,q)})$ and $\rot(K)$ where $K\in\mathcal L(\mathcal{K}_{(p,q)})$ with $ \tb(K)=\overline{\tb}(\mathcal{K}_{(p,q)})$. 
  \item Classify Legendrian knots with maximal Thurston-Bennequin invariant.
	\item Show that all Legendrian representatives of $\mathcal{K}_{(p,q)}$ of non-maximal Thurston-Bennequin invariant admit destabilization or determine those that cannot be destabilized.
	\item Understand the relationship between the stabilizations of two non-destabilizable representatives of $\mathcal{K}_{(p,q)}$.
	\end{enumerate}
 
\subsection{Sufficiently positive cables.} We will work our way up to the proof of Theorem~\ref{thm:positive} through a series of lemmas.

 \begin{lemma}\label{width rot and tb}
 Under the hyphothesis of Theorem \ref{thm:positive} the maximal Thurston-Bennequin invariant is $\overline{\tb} (K_{(p,q)}) = pq - |\overline{\tb} (\mathcal K) \bullet \frac{p}{q} |$. The set of rotation numbers realized by $L \in \mathcal K_{(p,q)} $ with  $\tb (L) = \overline{\tb}$ is 
\[
\rot (L) = \{q\cdot r (K) |\ K \in {\mathcal L} (\mathcal K) ,  \displaystyle  \tb (K) = \overline{\tb} (K) \}.
\]

\end{lemma}

\noindent

 \begin{proof} During the proof we will use the $\mathcal{C'}$ coordinate system. Note that $ tw(L,\mathcal{C'}) < 0$ for all $ L \in {\mathcal {L}}(K_{(p,q)})$. If not, we can assume there is $L'\in {\mathcal {L}} (K_{(p,q)})$ with $tw(L') = 0$. Then 
there exists a solid torus $S$ with $\partial S$ convex such that $L'$ is a Legendrian divide on $\partial S$ which  implies that slope of dividing set is $q/p$ when measured with respect to $\mathcal{C}$ but this contradicts the assumption that 
$\frac{p}{q} > \omega (\mathcal{K})$ 

Thus, there exists a solid torus $S$ representing $\mathcal{K}$  with $\partial S$ convex,
$L \subset \partial S$ and  the slope of $\Gamma_{\partial S}$ equal to $s.$

Recall in our Theorem \ref{thm:positive} it is assumed that $\omega(\mathcal K)\in\mathbb{Z}$. Since $\overline{ \tb}(\mathcal{K})\leq\omega (\mathcal{K})\leq \overline{ \tb}(\mathcal{K})+1$. We have either $\omega (\mathcal{K})= \overline{ \tb}(\mathcal{K})$ or $\omega (\mathcal{K})= \overline{ \tb}(\mathcal{K})+1$. Hence there are two cases to check.

\textit{Case 1.} $\omega (\mathcal{K})= \overline{ \tb}(\mathcal{K})$: We claim the following inequality holds under the assumptions of Theorem ~ \ref{thm:positive}
\begin{equation}\label{key} 
| \frac{1}{s} \bullet \frac{p}{q} | \geq | \omega (\mathcal K) \bullet \frac{p}{q} |
\end{equation}
and equality holds iff $\frac{1}{s} = \omega (\mathcal{K})$.

 To see this note that, since $\omega(\mathcal{K})\in \mathbb{Z}$ we know that on the Farey tesellesion there is an edge from $0$ to $\frac{1}{\omega (\mathcal{K})}$. Moreover, by definition of the contact width  we have, $\frac{1}{s}< \omega(\mathcal K)$. Now by using the oriented diffeomorphism of $\partial S$, we can normalize the  slopes by sending $0$ to $0$ and $\frac{1}{\omega (\mathcal K)}$ to $\infty$. Such a diffeomorphism will preserve order and hence force $q'/p' > 0$  and $\frac{1}{s'}\in[-\infty,0)$ where $q'/p'$ and $\frac{1}{s'}$ denotes the images of $q/p$ and $\frac{1}{s}$ under this diffeomorphism, respectively.

 Observe that $\frac{1}{s'}\in(-\infty,0)$ means 
\[ 
\frac{1}{s'}=m\left(\begin{array}{cc}    
0\\
-1 \end{array}\right)+n\left(\begin{array}{cc}    
1\\
0 \end{array}\right)=\left(\begin{array}{cc}    
n\\
-m \end{array}\right) 
\]

where $n,~m>0$. Hence as slope $\frac{1}{s'}=-\frac{m}{n}$. Now we easily get Inequality~\eqref{key}

\[
|\frac{1}{s}\bullet\frac{p}{q}|=|\frac{1}{s'}\bullet\frac{p'}{q'}|=|\frac{-m}{n}\bullet\frac{p'}{q'}|=|p'n+q'm|>q'=|\frac{-1}{0}\bullet\frac{p'}{q'}|=|\frac{1}{s}\bullet\frac{p}{q}|.
\]

Therefore $t(L,\mathcal C') \leq  -| \omega (\mathcal{K}) \bullet \frac{p}{q} |$. Now any Legendrian ruling on $\partial S$, where $S$ is solid torus representing $\mathcal{K}$ of maximal thickness (i.e. slope$\Gamma_{\partial S}=\frac{1}{\overline{\tb}(\mathcal K)}$), realizes the equality. By Equation \eqref{twisting} we see that
$$\overline{ \tb} (\mathcal K_{(p/q)} ) = pq - \left | \omega (\mathcal K) \bullet p / q \right|=pq-\left|\overline{ \tb}(\mathcal K)\bullet p/q\right|.$$
 
\textit{Case 2.} $\omega (\mathcal{K})= \overline{ \tb}(\mathcal{K})+1$: The same proof as in \textit{Case 1} is true when $s<\frac{1}{\overline{ \tb}(\mathcal K)+1}$ except in  Inequality~\eqref{key} equality holds iff $\frac{1}{s}=\overline{ \tb}(\mathcal K)$. When $s\in[\frac{1}{\overline{ \tb}(\mathcal K)},\frac{1}{\overline{ \tb}(\mathcal K)+1}]$, then first observe that for any such $s\in[\frac{1}{\overline{ \tb}(\mathcal K)},\frac{1}{\overline{ \tb}(\mathcal K)+1})$ we have
\begin{equation}\label{key1}
| \frac{1}{s} \bullet \frac{p}{q} | \geq | \overline{ \tb} (\mathcal K) \bullet \frac{p}{q} |.
\end{equation}  
Moreover, we cannot have $s=\frac{1}{\overline{ \tb}(\mathcal K)+1}$ as otherwise we would have $L\in \mathcal L(\mathcal K)$ with $\tb(L)=\overline{\tb}(\mathcal K)+1$.
 
 Therefore $t(L,\mathcal C') \leq -| \overline{\tb}(\mathcal{K}) \bullet \frac{p}{q} |$ and any Legendrian ruling curve of slope $q/p$ on $\partial N$, where  $N$ is solid torus representing $\mathcal K$ convex boundary and $s(\Gamma_{\partial N})=\frac{1}{\overline{ \tb}(\mathcal K)}$ will realize the equality in Inequality~\eqref{key1} 
 
  Next we compute the rotation numbers associated to this representatives. Take $L \in {\mathcal L} (\mathcal K_{(p,q)} )$ with $ \tb (L) = \overline{ \tb} (\mathcal K_{(p,q)})$. Then there exist a solid torus $S$ with convex boundary, where \mbox{slope$(\Gamma_{ \partial S}) = \frac{1}{\overline {\tb} (\mathcal K)}$} 
and $L$ is Legendrian ruling curve on $\partial S$.

Such a solid torus is a standard neighborhood of Legendrian knot \mbox{$K \in {\mathcal L} (\mathcal K)$}. Thus by Formula ~ \eqref{rottt} we have  
\[
\rot(L) = p \cdot \rot (\partial D) + q \cdot \rot(K) = q\cdot \rot(K)
\]
as $\rot(\partial D)=0$. 

 \end{proof}

 \begin{lemma} \label{width max tb class}
 
 The $L \in {\mathcal L} (\mathcal K_{(p,q)} )$ with $ \tb(L) = \overline{ \tb}$ are classified by their rotation numbers.

\end{lemma}

\begin{proof} If $L, L' \in {\mathcal L} (\mathcal K_{(p,q)})$ with $ \tb (L) =  \tb (L' ) = \overline{ \tb}$, then there exist solid tori $S$ and $S'$ which represent $K, K' \in {\mathcal L} (\mathcal K)$, respectively. Since \mbox{$tw(L,\partial S)<0$} (similarly $tw(L',\partial S')<0$) we can make $\partial S$ (and $\partial S'$) convex and $L$, $L'$ are Legendrian ruling curve on $S$ and $S'$, respectively. Moreover since $L$ and $L'$ are maximal  $\tb$ representatives there are only two dividing curves of slope $\frac{1}{\overline{\tb}(\mathcal K)}$ on  $\partial S$ and $\partial S'$.

If $\rot(L) =\rot(L')$, then by Lemma \ref{width rot and tb}, $\rot(K) =\rot(K')$ and hence 
$K$ and $K'$ are Legendrian isotopic by Legendrian  simplicity of the underlying knot type $\mathcal K$. Thus we may assume $K$ and  $K'$ are the same. Let $S$ and $S'$ be the standard neighborhoods of the $K=K'$ on which $L$ and $L'$, respectively, sit. Since $K=K'\subset S\cup S'$, there exist a solid torus $S''$ sitting inside both $S$ and $S'$ and with $\partial S''$ convex and slope($\Gamma_{\partial S''}) = \frac{1}{\overline{\tb} (\mathcal K)}$. Since $\overline{S-S''}$ and $\overline{S'-S''}$ are $I$-invariant neighborhoods, we can assume $L$, $L'$ are (slope $q/p$) Legendrian rulings on 
$\partial S''$. Finally, $L$ and $L'$ are Legendrian isotopic through the other Legendrian rulings.
\end{proof}

\begin{remark}    
If the knot type $\mathcal K$ satisfies UTP property, then 
 either there is single representative at maximal $\tb$ (hence has $\rot=0$) or several representatives at maximal $\tb$ which are distuinguished by  their rotation numbers. Since in our case we are dealing with the knot types that do not necessarily satisfy UTP, 
 there might be a knot type $\mathcal K$ that is Legendrian simple and has a Legendrian classification
such that some $K' \in {\mathcal L} (\mathcal K)$ has $\tb (K')=n < \overline{ \tb}$ but cannot be destabilized to 
$L$ with $ \tb (L) = \overline{ \tb}$. We note that Chongchitmate and Ng have conjectural examples in ~\cite{ChongchitmateNgPre} of this phenomena.   
\end{remark}
 

\begin{lemma}\label{nonmax nondestab} 
For each non-destabilizable $K \in \mathcal L(\mathcal K)$ with Thurston-Bennequin invariant  $\tb (K)=n < \overline{ \tb}$, there exists a unique, up to Legendrian isotopy, non-destabilizable $L$, a $(p,q)$-ruling curve on the standard neighborhood $N$ of $K$ with $\tb(L)=pq-| \frac{1}{n} \bullet \frac{p}{q} |$ and the set of rotation numbers associated to such $L$ is 

\[
\rot (L) = \{q\cdot \rot(K) |\ K \in {\mathcal L} (\mathcal K)\ , \  \tb (K) = n \}.
\]
\end{lemma}

\begin{proof}
  Let $K \in {\mathcal L} (\mathcal K)$ be such representative. Since $\tb (K)=n < \overline{ \tb}$ we can have an $L \in {\mathcal L} (K_{(p,q)})$ which is a Legendrian ruling on $\partial N'$ where 
$N$ is the standard neighborhood of $K \in {\mathcal L} (K)$ with $s(\Gamma_{\partial N}) = \frac{1}{n}$ and $n<\overline  \tb(\mathcal K)$. Now we want to show that $L$ does not admit a destabilization. Suppose that $L$ admits a destabilization. This implies the existence of a convex torus $\Sigma$ which is (topologically) isotopic to $\partial N$ and contains $L$ and a bypass for $L$. Now isotope the annulus $A=\partial N-L$ to $A'=\Sigma-L$ relative to the boundary $L$. By the Isotopy Discretization technique in \cite[Lemma 3.10]{h3},
we know such isotopy corresponds to a sequence of bypass attachments. Now we show that all potential bypass attachment are trivial, that is dividing set of $A$ will not change and hence we cannot reach $A'$. To end this, observe that a nontrivial bypass attachment from the outside will corresponds to a thickening of $\partial N$ and it cannot be thickened to some solid torus $N'$ with $s(\Gamma_{\partial N'}) = \frac{1}{n+1}$ since this will corresponds to a destabilization of $K \in {\mathcal L} (K)$ which is impossible. Hence a nontrivial bypass attachments will give a thickening of $\partial N$ to some solid tori $N'$ with $s(\Gamma_{\partial N'}) = s$ where $\frac{1}{n+1}<s<\frac{1}{n}$. An important observation is that since bypass attachment happens in the complement of $L$, any bypass attachments to $A$ cannot increase the intersection number of the dividing set with $L$. On the other hand, as in Case 1 in Lemma~\ref{width rot and tb}, one can easily show
\begin{equation}\label{bypass}
\biggl| \frac{q}{p} \bullet s \biggr| > \biggl|\frac{q}{p} \bullet \frac{1}{n} \biggr|.
\end{equation}
Thus, bypass attachment to $A$ from the outside must increase intersection number of the dividing set with $L$. Similarly bypass attachment to $A$ from the inside would increase the intersection of the dividing set with $L$. Hence, we cannot reach $A'$ and so $L$ does not destabilize
 \end{proof}

\begin{lemma} \label{destab or die}
If $L \in {\mathcal L} ({\mathcal K}_{(p,q)})$ with $ \tb(L) < \overline{ \tb} (K_{(p,q)})$, then either $L$ admits a destabilization or $L$ is one of the non-destabilizable representative from Lemma \ref{nonmax nondestab}.

\end{lemma}

\begin{proof} Given such an $L$ there is a solid torus $S$ representing $K$ with convex boundary, containing $L$ and dividing slope $s$. If $L$ does not intersect the dividing set $\Gamma_{\partial S}$ efficiently, then we can destabilize $L$ with a bypass on $\partial S$. So we now assume $L$ intersects $\Gamma_{\partial S}$ efficiently.  We know $s \neq \frac{1}{\omega(\mathcal K)}$, since $ \tb (L) < \overline{ \tb}(\mathcal K_{(p,q)})$. If $S$  has boundary slope $ \frac{1}{n}$, then either $K\in\mathcal L(\mathcal K)$ is non-destabilizable and we are in situation of Lemma \ref{nonmax nondestab} or, as the underlying knot type $\mathcal K$ is Legendrian simple, $K\in\mathcal L(\mathcal K)$ admits a destabilization and hence get a thickening of $S$. Now we can take a convex annulus $A=L\times[0,1]$ in $\partial S\times[0,1]$ and using the Imbalance Principle, we get a destabilization for $L$. Finally, suppose $s(\Gamma_{\partial S})=s$ and $S$ is non thickenable. Shrink $S$ to a solid torus $N'$ with $\partial N'$ convex and $s(\Gamma_{\partial N})=\frac{1}{n'}$. By using Equation~ \eqref{bypass} we get that  $|q'/p'\bullet s|=|q'/p'\bullet(-n/m)|=|p'n+q'm|>|p'n-q'nn'|>|p'-q'n'|=|q'/p'\bullet\frac{1}{n'}|$. Thus, we again get a destabilization for $L$. 
\end{proof}

 Finally we want to show for pairs $(\tb, \rot)$ obtained from stabilizations of multiple different non-destabilizable Legendrian knots (i.e. maximal $\tb$ representatives or Legendrian knots from Lemma~\ref{nonmax nondestab}), there is unique Legendrian with that $\tb$ and $\rot$. More precisely we prove

\begin{lemma}\label{positive valley} 
If $L, L' \in {\mathcal L} (\mathcal K_{(p,q)})$ with $ \tb (L) =  \tb (L')=\overline{\tb}(\mathcal K_{(p,q)})$ and $\rot(L)=\rot (L') + 2qn$, then $S^ {qn}_{-} (L)$ and $S^{qn}_{+} (L')$ are Legendrian isotopic. Also If $\tb(L)=\overline{\tb}(\mathcal K_{(p,q)})$ and $L'$ is from Lemma~\ref{nonmax nondestab} with $\rot(L)=\rot(L')+q(n-m)$, then and $S^ {qk}_{-}(L)$ and $S^{ql}_{+} (L')$, $k+l=n-m$, are Legendrian isotopic. 
\end{lemma}
\begin{proof} We need to show that  $S^ {qn}_{-} (L) = S^{qn}_{+} (L')$. Observe that $L$ and $L'$ sit on standard neighborhood of $K$ and $K'$, respectively, where $K$ and $K'$ of $\mathcal L(\mathcal K)$ have maximal $ \tb$ and $\rot(K)=\rot(K')+2n$, by the assumption and Lemma~\ref{width rot and tb}. As $\mathcal K$ is Legendrian simple, we have $S^n_{-} (K)=S^n_{+} (K')$. On the other hand since $L$ is in ${\mathcal L} (\mathcal K_{(p,q)})$ is Legendrian ruling curve of slope $\frac{q}{p}$ on the standard neighborhood, say $N(K)$, of $K$ in which we have the standard neighborhood, $N(S_{(K)})$, of $S_{(K)}$. Let $L_0$ be a Legendrian ruling curve of slope $q/p$ on $\partial N(S_{(K)})$ and let $A$ be a convex annulus between $N(K)$ and $N(S_{(K)})$ with $L$ and $L'$ being its boundary. A quick computation of $\tb$ shows that the dividing set on $A$ has to have $q$-boundary parallel arcs on $L_0$ side and no boundary parallel arcs on $L$ side (as otherwise we would be able to isotop $L$ along this bypass disks and end up with a representative with less twisting and contradict with the maximality of $L$). Now the boundary parallel arcs on $L_0$ side are all either positive or all negative, giving two kinds of destabilization of $L_0$.
Therefore, we can easily conclude that $S^q_{-}(L)$ sits on a standard neighborhood of $S_{-}(K)$. In a similar way $S^q_{+}(L')$ sits on the standard neighborhood of $S_{+}(K')$. One can induct this argumennt to see that $S^{qn}_{-}(L)$ and $S^{qn}_{+}(L')$ sit on the standard neighborhoods of $S^{n}_{-}(K)=S^{n}_{+}(K')$. Using the arguments as in the proof of Lemma~\ref{width max tb class}, we conclude that $L$ and $L'$ are Legendrian isotopic.

 By using similar argument we see can see that $L, L' \in {\mathcal L} (\mathcal K_{(p,q)})$ with $tb(L)=\overline{\tb}(\mathcal K_{(p,q)})$ and  $L'$ is from Lemma~\ref{nonmax nondestab} and $\rot(L)=\rot (L')+q(n-m)$ stabilizes to same Legendrian knot. 
 
\end{proof}
\begin{proof}[Proof of Theorem \ref{thm:positive}]
  Lemma \ref{width max tb class}, Lemma \ref{nonmax nondestab} and Lemma \ref{destab or die} give a complete list of non-destabilizable Legendrian knots in $\mathcal K_{(p,q)}$ and they are all determined by $ \tb$ and $\rot$, by Lemma~\ref{positive valley} 
\end{proof}
\subsection{Sufficiently negative cables.} Now we give the  proof of Theorem \ref{thm:lower}. The proof is established through the following sequence of lemmas.
   
\begin{lemma} \label{lower rot and tb}
  If $\frac{p}{q} < \ell \omega (\mathcal{K})$ and  $\ell \omega (\mathcal{K})\in \mathbb{Z}$, then $$ \displaystyle \overline{ \tb}(\mathcal{K}_{(p,q)})=pq=\omega(\mathcal{K}_{(p,q)}). $$
   Moreover the set of rotation numbers realized by$$ \{L_{(p,q)}\in\mathcal{L}(\mathcal{K}_{(p,q)}): \tb(L)=\displaystyle\overline{ \tb}(\mathcal{K}_{(p,q)})\}$$ is $$
   \{\pm(p+q(n+r(L)):L\in\mathcal{L}(\mathcal{K}), \tb(L)=-n\}$$ where n is the integer that satisfies $$-n-1<\frac{p}{q}<-n.$$
   
 \end{lemma}
 \begin{proof}
 We will use the $\mathcal{C'}$ coordinate system. Observe that since $\frac{p}{q} < \ell \omega (\mathcal{K})$, there is a convex torus of slope $q/p$, parallel to $\partial N$, inside solid torus $N$ representing $\mathcal K$, with convex boundary. Now a Legendrian divide on this convex torus is a representative $L_{(p,q)}\in\mathcal{L}(\mathcal{K}_{(p,q)})$ with twisting number zero. Thus $\overline{t}(L_{(p,q)},\mathcal{C'})\geq 0$.
  
  For the equality it is enough to show that $\omega(\mathcal{K}_{(p,q)},\mathcal{C'})=0$ since $\overline{t}(L_{(p,q)},\mathcal{C'})\leq \omega(\mathcal{K}_{(p,q)},C')$. The proof below is essentially the same as Claim 4.2 in \cite{eh1}. The key point is showing that the knot type $\mathcal{K}_{(p,q)}$ satisfies the first condition of the UTP.
  
  Let $N_{(p,q)}$ be a solid torus representing $\mathcal K_{(p,q)}$ and has convex boundary with $s(\Gamma_{\partial(N_{(p,q)}})=s$. We want to show $s=0$. Suppose $s>0$. After thinning the solid tori $N_{(p,q)}$ we may take $s$ to be a large positive integer and $\#\Gamma_{\partial(N_{(p,q)}}=2$. We use Giroux's Flexibility Theorem, \cite{gi1}, to arrange charecteristic foliation on $\partial N_{(p,q)}$ to be in standart form with Legendrian ruling of slope $\infty$ and consider convex annulus $A$ with Legendrian boundary of slope $\infty$ on $\partial N_{(p,q)}$ such that a thickening $R=N_{(p,q)}\cup (A\times [-\epsilon,\epsilon])\cong T^2\times [1,2]$ has  $\partial R=T_{1}\cup T_{2}$ parallel to $N(\mathcal K)$, where $N(\mathcal K)$ is a solid torus representing $\mathcal K$ with convex boundary of slope $q/p$, $T_{2}$ is isotopic to $\partial N$ and $T_{1}\subset N(\mathcal K)$. Note that $\Gamma_{A}$ must consists of parallel non-seperating arcs, otherwise we can attach the bypass corresponding to boundary parallel arcs onto $\partial(N_{(p,q)})$ to increase $s$ to $\infty$ by Theorem~\ref{bpaferry}. This will result excessive twisting inside $N(\mathcal K_{(p,q)})$ and hence would result contact structure to be overtwisted. Moreover, we can take an identification of $\partial N(\mathcal K)$ so that $slope(\Gamma_{T_{1}})=-s$ and $slope(\Gamma_{T_{2}})=1$. To see this, we note that $T_{1}$ and $T_{2}$ are each obtained by gluing one half of $\partial N(\mathcal K_{(p,q)})$ to the annulus $A$ and now since $s$ is a positive integer, it is clear that $\Gamma_{T_{1}}$ is obtained from $\Gamma_{T_{2}}$ by performing $s+1$ right-handed Dehn twists.
  
  Let $N'$ be a solid torus of maximal thickness containing $R$. By ~\cite[Proposition 4.1]{h1}, such a neighborhood has exactly two universally tight contact structures. On the other hand, any tight contact structure on $R$ can be layered into two basic slices at the torus $T_{1.5}$ parallel to $T_{i}$, $i=1,2$, with $slope(\Gamma_{T_{1.5}})=\infty$ which is $q/p$ when measured with respect to $\mathcal{C}$ coordinate system. Moreover, a quick computation of the Poincare duals of the relative Euler classes for each of this basic slices shows that there are four possible tight contact structures on $R$ (two for each basic slices) which are given by $\pm(1,0)\pm (1,1-s)$ and the universally tight ones are the ones that has no mixing of sign (i.e. either $+(1,0)+(1,1-s)$ or $-(1,0)-(1,1-s)$ ). We want to determine if the tight contact structure $\xi$ we start with, has a mixing of sign or not. To end this, we compute the Euler class. Let $\gamma$ be a Legendrian ruling curve of slope $\infty$ on $A$ and let $A'=\gamma\times[-\epsilon,\epsilon]$. We easily see that the dividing set on $A'$ is made of $2s$ parallel curves (as $A'$ is $(-\epsilon,\epsilon)$--invariant), we use this to get that $<e(\xi),A'>=\chi(A'_{+})-\chi(A'_{-})=0$, this gives then $PDe(\xi)=\pm(0,1-s)$. So, there is a mixing of sign. But this cannot happen inside $N'$. Thus, $s=0$ and we get $\omega(\mathcal{K}_{(p,q)},\mathcal{C'})=0$, passing $\mathcal{C}$ coordinate system we have $\displaystyle \overline{ \tb}(\mathcal{K}_{(p,q)})=pq$.

 Now we want to compute rotation numbers of $L_{(p,q)}$ in $\mathcal L(\mathcal K_{(p,q)})$ realizing maximal Thurston-Bennequin number. Let $T^2_{1.5}=\partial N$ which contains $L_{(p,q)}$ with $ \tb(L_{(p,q)})=pq$. Since $\frac{p}{q}< \ell \omega (\mathcal{K})$, we can take a thickening of tori $T^2_{1.5}$,  $T^2 \times [1,2]$ such that boundary tori have slope \mbox{slope$(\Gamma_{T^2_{1}})=-\frac{1}{n-1}$} and slope$(\Gamma_{T^2_{1}})=-\frac{1}{n}$ where $n$ is the integer that satisfies $-n-1<\frac{p}{q}<-n$ (note that $n$ may equal to $\ell \omega (\mathcal{K})$). But now the solid tori of boundary slopes $-\frac{1}{n-1}$ and $-\frac{1}{n}$ are the standard neighborhoods of  $L$ and $S_{\pm} (L)$, respectively. We can now make the relative Euler class computation as above and then use Lemma~\ref{rotation} to get desired formula for the rotation number computation.
\end{proof}
 
\begin{lemma}\label{rot det}
 Legendrian knots with maximal $ \tb$ in  $\mathcal L(\mathcal K_{(p,q)})$ are determined by their rotation numbers.  
\end{lemma} 

\begin{proof}
Let $L$ and $L'$ be two Legendrian knots in $\mathcal L(\mathcal K)$ with maximal $ \tb$ and $\rot(L)=\rot(L')$, then we have associated solid tori $N$ and $N'$ with convex boundary on which  $L$ and $L'$ sit as Legendrian divides. The classification of tight contact structures on the solid torus in \cite{{gi2},{h1}} says that the contactomorphism type of a tight contact structure on a solid torus with convex boundary is determined by the number of the positive bypasses on the meridional disk. Hence, determined by the rotation number of $L$ and $L'$, respectively, which are the same by the assumption. Thus, we get a contactomorphism $f:N\rightarrow N'$. We may extend $f$ to a contactomorphism of $S^3$ that takes $\partial N$ to $\partial N'$. Furthermore, by using Eliashberg's result in \cite{el}, there is a contact isotopy of $S^3$ that takes $\partial N$ to $\partial N'$. So we will now think $L$ and $L'$ are Legendrian divides on same solid torus, say $N$, with convex boundary.  We now want to form a Legendrian isotopy between $L$ and $L'$. To end this, we recall from Lemma ~\ref{lower rot and tb} that $\partial N$ is siting inside a thickened torus  $T^2 \times [1,2]$ such that boundary tori have  \mbox{slope$(\Gamma_{T^2_{1}})=-\frac{1}{n-1}$} and slope$(\Gamma_{T^2_{2}})=-\frac{1}{n}$. Now as the consequence of the classification of tight contact structure on thickened tori (see ~\cite[Corollary 4.8]{h1}), we know there is also a pre-Lagrangian torus, (still) denote by $\partial N$, which has linear characteristic foliation and the same boundary slope as convex torus does. Thus, we can take $L$ and $L'$ to be two leaves on this pre-Lagrangian torus. Now, $L$ and $L'$ are Legendrian isotopic through this linear characteristic foliation.
\end{proof} 
  
 \begin{lemma}\label{lower destab}
 If $L'\in\mathcal{K}_{(p,q)}$ with ${ \tb}(L')<\overline{ \tb}$, then $L'$ admits a destabilization.
 \end{lemma}
 
 \begin{proof}
 We can put $L'$ on a solid torus $S$ with $\partial S$ convex and \mbox{ $slope(\Gamma_{\partial S})=s$}. By the above lemma and the assumption that $\frac{p}{q}< \ell \omega (\mathcal{K})$ we can deduce that $L'$ is a Legendrian ruling on $S$ (clearly we can assume $L'$ intersects $\Gamma_{\partial S}$ efficiently otherwise destabilization is immediate) and $\frac{1}{s}\neq\ell \omega (\mathcal{K})$. If $s<\frac{1}{\ell\omega(\mathcal K)}$, then, as in Equation~\eqref{width rot and tb}, we easily see that $|q/p\bullet s|>|q/p\bullet1/\ell\omega|$. Hence,by using the Imbalance Principle, we get a destabilaztion of $L'$. If $s>\frac{1}{\ell\omega(\mathcal K)}$, then we can thicken $S$ to a solid tori $S'$ with $\partial S'$ convex and \mbox{ $slope(\Gamma_{\partial S'})=\frac{1}{\ell \omega (\mathcal{K})}$}. Hence taking a convex annulus $A$ with one boundary component on $L'$ in $\partial S\times[0,1]=\overline{S'-S}$ and applying the Imbalance Princible again we find a bypass for $L'$ which gives a destabilization for $L'$.
 \end{proof} 

\begin{lemma}\label{peaks}
  If $L^{+}_{(p,q)},L^{-}_{(p,q)}\in \mathcal L (\mathcal K_{(p,q)})$ with $ \tb(L^{+}_{(p,q)})= \tb(L^{-}_{(p,q)})$ and $\rot(L^{+}_{(p,q)})=\rot(L^{-}_{(p,q)})+2p+2qn$ (or $\rot(L^{+}_{(p,q)})=\rot(L^{-}_{(p,q)})+2kq-2p-2qn$), then $S^{p+qn}_{+}(L^{-}_{(p,q)})=S^{p+qn}_{-}(L^{+}_{(p,q)})$ (or $S^{kq-p-qn}_{+}(L^{-}_{(p,q)})=S^{kq-p-qn}_{-}(L^{+}_{(p,q)})$).
\end{lemma}
 
 \begin{proof}
There are two cases to concern based on rotation number computation in Lemma ~\ref{lower rot and tb}

\dfn{Case 1:} $L\in \mathcal L(\mathcal K)$ in Lemma ~\ref{lower rot and tb} has $\rot(L)=0$. In this case $L^{\pm}_{(p,q)}$ are the only maximal $\tb$ representatives of $\mathcal L(\mathcal K_{(p,q)})$ with $r(L^{+}_{(p,q)})=-p-qn$ and $r(L^{-}_{(p,q)})=p+qn$. Clearly by doing $-p-qn$ positive ( respectively negative) stabilization on $L^{-}_{(p,q)}$ (respectively on $L^{+}_{(p,q)}$) we end up at Legendrian knots with the same $(\tb, \rot)$ pair. We also have $L'_{(p,q)}\in\mathcal L (\mathcal K_{(p,q)})$ with $\tb(L'_{(p,q)})=\displaystyle\overline{ \tb}(\mathcal{K}_{(p,q)})\}+p+qn$ number and $\rot(L'_{(p,q)})=q\rot(L)=0$. We know by Lemma~\ref{lower destab}, such a $L'_{(p,q)}$ admits a destabilization. We want to show, these are Legendrain isotopic, i.e. $S^{-p-qn}_{+}(L^+_{(p,q)})=L'_{(p,q)}=S^{-p-qn}_{-}(L^-_{(p,q)})$. Recall that $L^{\pm}_{(p,q)}$ are the Legendrian divide on a convex torus $T_{1.5}$ with boundary slope $\frac{q}{p}$ inside $T^2\times [1,2]=N(L)- N(S_{\pm}(L))$ (See the remark at the end of the statement of  \cite[Lemma 3.8]{eh1}). Hence $L'_{(p,q)}$ is a Legendrian ruling curve of slope $\frac{q}{p}$ on the standard neighborhood $N(L)$ of $L\in \mathcal L(\mathcal K)$ with $ \tb(L)=-n$. Note that, $S^{-p-qn}_{+}(L^+_{(p,q)})$ and $S^{-p-qn}_{-}(L^-_{(p,q)})$ are also Legendrian ruling curve on $N(L)$. Hence, $L'_{(p,q)}$ is Legendrian isotopic to  $S^{-p-qn}_{+}(L^+_{(p,q)})$ and $S^{-p-qn}_{-}(L^-_{(p,q)})$  through the other ruling curves. Indeed, by taking a convex annulus $A=L_{(p,q)}\times [1.5,2]$ between $T_{1.5}$ and $N(L)$ with $\partial A$ is Legendrian curves of slope $\frac{q}{p}$ on $T_{1.5}$ and $N(L)$, we easily see $L'_{(p,q)}$ destabilizes in two ways.

\dfn{Case 2:} $L\in \mathcal L(\mathcal K)$ in Lemma ~\ref{lower rot and tb} has $\rot(L)\neq0$. In this case, $L^{\pm}_{(p,q)}\in\mathcal L(\mathcal K_{(p,q)})$ coresponds to $L^{\pm}\in \mathcal L(\mathcal K)$ where $ \tb(L^{+})= \tb(L^{-})=-n$ and $r(L^{+})\neq r(L^{-})$. Without loss genarility we can assume that $\rot(L^{-})< \rot(L^{+})$ and there is no $L^0$ with $\rot(L^{-})<\rot(L^0)< \rot(L^{+})$ , then $\rot(L^{+})- \rot(L^{-})=2k$, $k\in \mathbb{Z}_{>0}$. Thus $\rot(L^{-}_{(p,q)})=q\rot(L^{-})+p+qn$ and $\rot(L^{+}_{(p,q)})=q\rot(L^{+})-p-qn=q\rot(L^{-})-(2kq+p+qn)$. This extra depth $kq$ comes from the underlying knot type puts us precisely in the situation of Lemma ~\ref{positive valley}. Namely, the $L'_{(p,q)}$ with $\tb(L'_{(p,q)})=pq-(kq+p+qn)$ and $\rot(L'_{(p,q)})=q\rot(L^{+})+kq=q\rot(L^{-})-kq$ is the Legendrian ruling curve of slope $\frac{q}{p}$ on the standard neighborhood $S^k_{+}(L^{+})=S^k_{-}(L^{-})$ (as $\mathcal K$ is Legendrian simple). Therefore, a Legendrian isotopy through the other ruling curves gives that $L'_{(p,q)}=S^{kq-p-qn}_{+}(L^{-}_{(p,q)})=S^{kq-p-qn}_{-}(L^{+}_{(p,q)})$.
\end{proof}

\begin{proof}[Proof of Theorem \ref{thm:lower}]
  Lemma \ref{lower rot and tb} and Lemma~\ref{rot det} give a complete list of non-destabilizable Legendrian knots in $\mathcal K_{(p,q)}$ and show they are all determined by their $ \tb$ and $rot$. By Lemma~\ref{lower destab}, every $L'_{(p,q)}$ in $\mathcal L(\mathcal K_{(p,q)})$ with non-maximal $\tb$ invariant can be written as $S^{k}_{-}S^{l}_{+}(L_{(p,q)})$ for some $L^{\pm}_{(p,q)}\in\mathcal L(\mathcal K_{(p,q)})$ with maximal $\tb$. Finally, Lemma~\ref{peaks} shows any two $L^{\pm}_{(p,q)}$ with maximal $\tb$ and $\rot(L^{-}_{(p,q)})< \rot(L^{+}_{(p,q)})$ (and no $L^{0}_{(p,q)}$ with $\rot(L^{-}_{(p,q)})< \rot(L^{0}_{(p,q)})<\rot(L^{+}_{(p,q)})$ ), stabilize to same $L'_{(p,q)}$ in $\mathcal L(\mathcal K_{(p,q)})$.
\end{proof} 

\bibliographystyle{amsplain}

\end{document}